\documentclass[reqno, english]{amsart}
\usepackage{amsmath,amssymb,amsthm,bbm,mathtools,comment}
\usepackage[shortlabels]{enumitem}
\usepackage[pdftex,colorlinks,backref=page,citecolor=blue]{hyperref}
\hypersetup{pdfpagemode=UseNone,pdfstartview={XYZ null null 1.00}}
\usepackage[mathscr]{euscript}
\usepackage[usenames,dvipsnames]{color}
\usepackage{adjustbox,tikz,calc,graphics,babel,standalone}
\usetikzlibrary{shapes.misc,calc,intersections,patterns,decorations.pathreplacing}
\usetikzlibrary{arrows,shapes,positioning}
\usetikzlibrary{decorations.markings}
\usepackage[final]{microtype}
\usepackage[numbers,sort]{natbib}
\usepackage{amsfonts}
\usepackage{graphicx}
\usepackage{caption}
\usepackage{subcaption}
\usepackage{verbatim}
\usepackage{array}
\usepackage[frame,cmtip,arrow,matrix,line,graph,curve]{xy}
\usepackage{graphpap, color, pstricks}
\usepackage{pifont}
\usepackage[final]{microtype}
\usepackage{todonotes}
\usepackage[capitalize]{cleveref}
\usetikzlibrary{topaths,calc} 
\tikzstyle{vertex} = [fill,shape=circle,node distance=80pt]
\tikzstyle{edge} = [opacity=0.4,fill opacity=0.0,line cap=round, line join=round, line width=40pt]
\tikzstyle{elabel} =  [fill,shape=circle,node distance=30pt]

\allowdisplaybreaks

\theoremstyle{plain}
\newtheorem{theorem}{Theorem}[section]		
\newtheorem{lemma}[theorem]{Lemma}

\newtheorem{conjecture}[theorem]{Conjecture}

\theoremstyle{remark}

\newcommand{\e}{\ensuremath{\varepsilon}}

\newcommand{\hide}[1]{}

\let\originalleft\left
\let\originalright\right
\renewcommand{\left}{\mathopen{}\mathclose\bgroup\originalleft}
\renewcommand{\right}{\aftergroup\egroup\originalright}

\makeatletter
\def\imod#1{\allowbreak\mkern10mu({\operator@font mod}\,\,#1)}
\makeatother

\title{Restricted partitions of graphs with bounded clique number}

\author{Ant\'onio Gir\~ao}
\author{Toby Insley}

\begin{document}

\title{Sparse Partitions of Graphs with Bounded Clique Number}

\maketitle
\begin{abstract}
    \noindent We prove that for each integer $r\ge 2,$ there exists a constant $C_r>0$ with the following property: for any $0<\e\le 1/2$ and any graph $G$ with clique number at most $r,$ there is a partition of $V(G)$ into at most $(1/\e)^{C_r}$ sets $S_1, \dots, S_t,$ such that $G[S_i]$ has maximum degree at most $\e|S_i|$ for each $1 \le i \le t.$ This answers a question of Fox, Nguyen, Scott and Seymour, who proved a similar result for graphs with no induced $P_4.$
\end{abstract}

\section{Introduction}

Given graphs $G$ and $H,$ as usual we say that $G$ is \textit{$H$-free} if $G$ contains no induced subgraph isomorphic to $H.$ What can be said about the structure of such graphs $G?$ The most famous open problem in this regard is a conjecture of Erdős and Hajnal~\cite{ErdosHajnal77} from 1977.

\begin{conjecture}\label{E-H conjecture}
    For every graph $H,$ there exists a constant $c_H>0$ such that any $H$-free graph $G$ contains either a clique or a stable set of size at least $|G|^{c_H}.$ 
\end{conjecture}

We say that any graph $H$ satisfying Conjecture \ref{E-H conjecture} has the \textit{Erdős-Hajnal property}. We remark that Erdős and Szekeres~\cite{Erdös1935} showed that all graphs $G$ contain either a clique or a stable set of size at least $\frac{1}{2}\log |G|,$ and Erdős~\cite{RemarkOnTheoryOfGraphs} further showed that almost all graphs $G$ contain neither a clique nor a stable set of size greater than $(2+o(1))\log|G|$ (all logarithms in this paper are to the base 2). In light of this, Conjecture \ref{E-H conjecture} would indicate that $H$-free graphs have a highly atypical structure.

Conjecture \ref{E-H conjecture} has received a lot of attention over the years, having now been verified in a good number of nontrivial cases. A natural variation of this problem asks the following: to what degree do $H$-free graphs contain large regions which are either very `sparse' or very `dense'? Given a graph $G$ and any $\epsilon>0,$ say that a set $S\subseteq V(G)$ is \textit{weakly \textit{$\epsilon$-restricted}} if one of $G[S],\overline{G}[S]$ contains at most $\epsilon|S|^2$ edges. Rödl~\cite{Rodl} showed that $H$-free graphs contain linear-sized such sets.

\begin{theorem}[Rödl]\label{Rodl, weak}
    For every graph $H$ and any $\epsilon>0,$ there exists $\delta>0$ with the following property: if $G$ is an $H$-free graph, then there is some subset $S\subseteq V(G)$ with $|S| \ge \delta |G|,$ such that $S$ is weakly $\epsilon$-restricted.
\end{theorem}

For a fixed graph $H,$ how large can we take $\delta$ as a function of $\epsilon?$ Fox and Sudakov~\cite{fox-sudakov} conjectured the following polynomial bound.

\begin{conjecture}\label{Fox-Sudakov conjecture}
    For every graph $H,$ there exists a constant $C_H>0$ with the following property: for all $0 < \epsilon \le 1/2$ and any $H$-free graph $G,$ there is some weakly $\e$-restricted set $S\subseteq V(G)$ with $|S| \ge \epsilon^{C_H}|G|.$
\end{conjecture}

We say that any graph $H$ satisfying Conjecture \ref{Fox-Sudakov conjecture} has the \textit{polynomial Rödl property}. An easy application of Turán's Theorem shows that any graph with the polynomial Rödl property also has the Erd\H{o}s-Hajnal property. It was recently shown in a beautiful paper of Bucić, Fox and Pham that the converse is also true.

\begin{theorem}
    Let $H$ be a graph. Then $H$ has the Erdős-Hajnal property if and only if $H$ has the polynomial Rödl property.
\end{theorem}

\bigskip

An interesting variation of the above problem asks to what degree we may partition $H$-free graphs into a `small' number of sets, all of which are either sparse or dense. Indeed, this has a straightforward answer in the context of what we have defined so far.

\begin{theorem}\label{Rodl partition, weak}
    For every graph $H$ and for all $\epsilon>0,$ there exists $N\ge 1$ such that any $H$-free graph $G$ partitions into at most $N$ sets, all of which are weakly $\epsilon$-restricted. Moreover, $H$ satisfies the polynomial Rödl property if and only if we may take $N=(1/\e)^{C_H}$ for some constant $C_H > 0.$
\end{theorem}

This can be seen by partitioning all but a very small number of vertices into weakly $(\e/2)$-restricted sets (by repeatedly applying Theorem \ref{Rodl, weak}, along with the relevant quantitative bounds for the `moreover' statement), and then placing the leftover vertices into a largest such one of these sets. 

What if we demand a little more restriction on the sets making up such a partition? Given a graph $G$ and any $\epsilon>0,$ say now that a set $S\subseteq V(G)$ is \textit{strongly $\epsilon$-restricted} if one of $G[S],\overline{G}[S]$ has maximum degree at most $\epsilon|S|.$ It is easily seen that any strongly $\epsilon$-restricted set is weakly $(\epsilon/2)$-restricted, and that any weakly $(\epsilon/2)$-restricted set contains a set of at least half its size which is strongly $\epsilon$-restricted. It follows that Theorem \ref{Rodl, weak} and Conjecture \ref{Fox-Sudakov conjecture} are equivalent to their `strong' analogies (the statement given by replacing the word `weakly' with `strongly'). However, it is not hard to see why the above argument for Theorem \ref{Rodl partition, weak} fails to adapt to the case of strongly restricted sets. Indeed, it was shown only recently in a paper of Chudnovsky, Scott, Seymour and Spirkl~\cite{StrongRodl} that $H$-free graphs can be partitioned into a bounded number of strongly $\e$-restricted sets.

\begin{theorem}\label{Rodl partition, strong}
    For every graph $H$ and for all $\epsilon>0,$ there exists $N\ge 1$ such that any $H$-free graph $G$ partitions into at most $N$ sets, all of which are strongly $\epsilon$-restricted.
\end{theorem}

A corresponding strengthening of Conjecture \ref{Fox-Sudakov conjecture} was raised by Fox, Nguyen, Scott and Seymour~\cite{Foxetalcographs}: given a graph $H,$ is there some constant $C_H>0$ such that we may take $N= (1/\e)^{C_H}$ in Theorem \ref{Rodl partition, strong}? Say that any graph $H$ for which this holds has the \textit{strong polynomial Rödl property}. It was shown in~\cite{Foxetalcographs} that $H=P_4$ has this property, and it was remarked that the same was not known when $H$ is a triangle. We address this case and, more generally, the case where $H$ is any complete graph. 

\begin{theorem}
    Every complete graph has the strong polynomial Rödl property.
\end{theorem}

\section{Proofs}

Given $\e>0$ and a graph $G,$ say that a set $S\subseteq V(G)$ is $\e$-sparse if $G[S]$ has maximum degree at most $\e |S|.$ Moreover, we say that $\mathcal{A}$ partitions $G$ if it consists of pairwise disjoint subsets of $V(G)$ such that $\cup_{A\in \mathcal{A}}A=V(G)$.

We will actually prove the following stronger result.

\begin{theorem}\label{main result}
    For each integer $r\ge 2,$ there exists a constant $C_r>0$ such that for any $0<\e\le 1/2,$ every $K_{r+1}$-free graph $G$ can be partitioned into at most $(1/\e)^{C_r}$ sets, all of which are $\e$-sparse.
\end{theorem}

We will use the following standard Chernoff bound.

\begin{lemma}\label{chernoff}
    Let $X$ be a binomially distributed random variable with mean $\mu.$ Then for all $0<\delta<1,$ the following holds.

    $$\mathbb{P}\left(|X-\mu| \ge \delta\mu\right) \le 2e^{-\delta^2\mu/3}.$$
\end{lemma}

We begin with the following easy lemma.

\begin{lemma}\label{sparse set}
    Let $r\ge 1$ be an integer. Then for all $0< \alpha \le 1,$ if $G$ is any $K_{r+1}$-free graph, there exists an $\alpha$-sparse set $S\subseteq V(G)$ of size at least $\alpha^{r-1}|G|.$
\end{lemma}

\begin{proof}
    Proceed by induction on $r\ge 1.$ For $r=1$ this is trivial, so suppose that $r\ge 2.$ If $V(G)$ is $\alpha$-sparse, then we are done; otherwise, there is some vertex $v\in V(G)$ with at least $\alpha|G|$ neighbours. Since $G[N(v)]$ must be $K_r$-free, applying the inductive hypothesis gives the result.
\end{proof}

We will use Lemma \ref{sparse set} both in its own right, and also in the following application.

\begin{lemma}\label{partition most}
    Let $r\ge 2$ be an integer and let $m \ge 0.$ Then for all $0 < \alpha \le 1/2,$ if $G$ is any $K_{r+1}$-free graph, there exists a partition $\mathcal{A} \cup \{L\}$ of $V(G),$ such that

    \begin{itemize}
        \item $|\mathcal{A}| \le (1/\alpha)^{m+2r-1},$ and each set in $\mathcal{A}$ is $\alpha$-sparse;
        \item $|L| \le \alpha^{(r-1)(1/\alpha)^m}|G|.$
    \end{itemize}
\end{lemma}

We will need the following approximation, the proof of which we omit.

\begin{lemma}\label{approx}
    Let $0< x < 1.$ Then for all $s\ge 1,$ we have that

    $$(1-x)^{(1/x)^s} \le x^{(1/x)^{s-2}}.$$
    
\end{lemma}

We now deduce Lemma \ref{partition most} as follows.

\begin{proof}[Proof of Lemma \ref{partition most}]
    Produce $\mathcal{A}$ by greedly applying Lemma \ref{sparse set} $\lfloor(1/\alpha)^{m+2r-1}\rfloor$ times, and call the leftover set of vertices $L.$ Then

    $$|L| \le (1-\alpha^{r-1})^{\lfloor(1/\alpha)^{m+2r-1}\rfloor}|G| \le (1-\alpha^{r-1})^{(1/\alpha)^{m+2r-2}}|G|.$$

    Now write $(1/\alpha)^{m+2r-2} = (1/\alpha^{r-1})^{\frac{m}{r-1}+2}.$ Then by applying Lemma \ref{approx} with $s=\frac{m}{r-1}+2,$ it follows that

    $$|L| \le (\alpha^{r-1})^{(1/\alpha^{r-1})^{\frac{m}{r-1}}}|G| = \alpha^{(r-1)(1/\alpha)^m}|G|,$$

    and so we are done.
\end{proof}

\bigskip

We now introduce some terminology to describe the distribution of edges between two sets of vertices. Given $\alpha,\beta,\gamma>0,$ a graph $G,$ and disjoint sets $A,B \subseteq V(G),$ we say that 

\begin{itemize}
    \item $B$ is $\alpha$-dense (resp. $\alpha$-sparse) to $A$ if every vertex in $B$ has at least (resp. at most) $\alpha|A|$ neighbours in $A;$
    \item the pair $(A,B)$ is $(\alpha,\beta,\gamma)$-full if, for every $X\subseteq A$ with $|X| \ge \alpha|A|,$ all but at most $\beta|B|$ of the vertices in $B$ are $\gamma$-dense to $X.$
\end{itemize}

More generally, we will say that a sequence of disjoint sets $S_1, \dots, S_r$ is $(\alpha, \beta, \gamma)$-full if the pair $(S_i,S_j)$ is $(\alpha, \beta, \gamma)$-full for each $1 \le i < j \le r.$ The following says that an appropriately full sequence of $r$ sets must contain a copy of $K_r.$

\begin{lemma}\label{sequence}
    Let $0<\gamma \le 1/2,$ let $r\ge1$ be an integer, let $G$ be a graph, and suppose that $S_1, \dots, S_r$ is sequence of disjoint, nonempty subsets of $V(G)$ which is $(\gamma^{r-1}, \gamma^{r-1}, \gamma)$-full. Then there must be a copy of $K_r$ spanning the sets $S_1, \dots, S_r.$
\end{lemma}

\begin{proof}
    Fix $\gamma$ as such and proceed by induction on $r.$ For $r=1,$ this is trivial, so suppose that $r\ge2,$ and that $S_1, \dots, S_r$ are subsets of some graph $G$ as per the statement. Now for each $1 \le i \le (r-1),$ by assumption, all but at most $\gamma^{r-1} |S_r|$ of the vertices in $S_r$ are $\gamma$-dense to $|S_i|.$ Therefore, at least $|S_r|(1-(r-1)\gamma^{r-1})>0$ of the vertices in $S_r$ are $\gamma$ dense to $S_i$ for every $1 \le i \le (r-1).$ Let $v \in S_r$ be some such vertex, and let $S_i' = S_i \cap N(v)$ for each $1 \le i \le (r-1),$ so then $|S_i'| \ge \gamma |S_i|$ for each $1 \le i \le (r-1).$ Then $S_1', \dots, S_{r-1}'$ is a sequence of disjoint, nonempty sets which is $(\gamma^{r-2},\gamma^{r-2},\gamma)$-full, and hence there is a copy of $K_{r-1}$ spanning these sets. Combining this with the vertex $v$ gives the required copy of $K_r.$
\end{proof}

The following lemma is key in the proof of Theorem \ref{main result}.

\begin{lemma}\label{pairs lemma}
    Let $0<\alpha,\beta \le 1,$ and let $l\ge1.$ Let $G$ be a graph, and suppose that $(A,B)$ is a disjoint pair of subsets of $V(G)$ such that $B$ is $\alpha$-dense to $A.$ Then there exist subsets $A' \subseteq A$ and $B' \subseteq B$ such that

    \begin{itemize}
        \item $|A'| \ge (\alpha/2) |A|;$
        \item $|B'| \ge \beta^{(1/\alpha)^{2l}}|B|;$
        \item the pair $(A',B')$ is $(\alpha^l,\beta,\alpha/2)$-full.
    \end{itemize}
\end{lemma}

\begin{proof}
    We may assume that $A$ and $B$ are nonempty, otherwise the result is trivial. Now suppose that $n \ge 0$ is an integer and that $A' \subseteq A$ and $B' \subseteq B$ are sets which satisfy the following.

    \begin{itemize}
        \item $|A'| \le (1-\alpha^l)^n|A|;$
        \item $|B'| \ge \beta^n|B|;$
        \item each vertex of $B'$ has at least $(\alpha/2)(|A|+|A'|)$ neighbours in $A'.$
    \end{itemize}

\noindent\textbf{Claim:} We must have that $n \le (1/\alpha)^{2l}.$

\noindent\textbf{Proof:} Since $B$ is nonempty, we must have that $B'$ is nonempty, and so certainly $|A'| \ge (\alpha/2)(|A|+|A'|) \ge (\alpha/2)|A|.$ But since also $|A'| \le (1-\alpha^l)^n|A|$ and $|A|>0,$ it follows that $\alpha/2 \le (1-\alpha^l)^n,$ which forces $n \le (1/\alpha)^{2l}.$ This proves the claim.

\bigskip

Now by the above (and since $A'=A$ and $B'=B$ satisfy the above with $n=0$), we may fix some $n,A',B'$ as above, such that $n$ is maximal. Because $|A'| \ge (\alpha/2)|A|,$ we may assume that the third bullet of the statement fails, so that there are sets $X\subseteq A'$ and $Y \subseteq B'$ with $|X| \ge \alpha |A'|$ and $|Y| \ge \beta |B'|,$ such that $Y$ is $(\alpha/2)$-sparse to $X.$ Now define $\tilde{A} = A' \backslash X$ and $\tilde{B} = Y.$ Then we have that

\begin{itemize}
    \item $|\tilde{A}| = |A'| - |X| \le |A'| - \alpha^l |A'| = (1-\alpha^l)|A'| \le (1-\alpha^l)^{n+1}|A|;$
    \item $|\tilde{B}| = |Y| \ge \beta|B'| \ge \beta^{n+1}|B|;$
    \item each vertex of $B_{n+1}$ has at least

    $$(\alpha/2)(|A|+|A'|) - (\alpha/2)|X| = (\alpha/2)(|A|+|\tilde{A}|)$$

neighbours in $\tilde{A}.$
\end{itemize}

But this contradicts the maximality of $n,$ and so we are done.

\end{proof}

\bigskip

Lastly, we need the following technical lemma.

\begin{lemma}\label{split}
    Let $\alpha>0$ and let $G$ be a graph. Suppose that $(A,B)$ is a disjoint pair of subsets of $V(G)$ such that 

    \begin{itemize}
        \item $|A| \ge 100;$
        \item $|B| \le (\alpha / 100) |S|;$
        \item $B$ is $\alpha$-dense to $A.$
    \end{itemize}
    
    Then there is a partition $A= T \cup T'$ such that $|T|,|T'| \ge |A|/3$ and such that $B$ is $(\alpha/2)$-dense to $T.$
\end{lemma}

We now deduce Lemma \ref{split} as follows.

\begin{proof}[Proof of Lemma \ref{split}]
    Let $T$ be a random subset of $A,$ where vertices are selected independently and uniformly at random. We claim that, with nonzero probability, the following conditions hold.

\begin{enumerate}[(i)]
    \item $|A|/3 \le |T| \le 2|A|/3;$
    \item $|N(v) \cap T| \ge \alpha|A|/3$ for each $v \in B.$
\end{enumerate}

Indeed, by applying Lemma \ref{chernoff} with $\delta = 1/3,$ the probability that (i) fails is at most $2e^{-|A|/54},$ and the probability that (ii) fails for some particular vertex $v$ is at most $2e^{-\alpha|A|/54}.$ Thus, the probability that either (i) or (ii) fails is at most

$$2\left(e^{-|A|/54} + |B|e^{-\alpha|A|/54}\right) \le 2(e^{-54/100}+54e^{-1}/100)<1$$

(where we have used that $xe^{-x} \le e^{-1}$ for all $x\ge 0$). Therefore, we may fix some choice of $T$ for which both (i) and (ii) hold. Define $T'=A\backslash T.$ Then certainly $|T|,|T'| \ge |A|/3,$ and also for every $v\in B,$ we have that

$$|N(v) \cap T| \ge \alpha|A|/3 \ge \alpha|T|/2.$$

Thus $B$ is $(\alpha/2)$-dense to $T,$ and so we are done.
\end{proof}

\bigskip

We are now ready to prove Theorem \ref{main result}.

\begin{proof}[Proof of Theorem \ref{main result}]
    Inductively define a decreasing sequence of positive numbers $a_0, \dots, a_{r-1},$ so that the following holds for each integer $0 \le i \le r-1$ and for all $x\ge 2.$ 

    $$x^{a_i} \ge \sum_{i=j+1}^{r-1} x^{a_j}(6x)^{26r} + 13r.$$

    Then define $C_r>0$ large enough such that the following holds for all $x \ge 2.$

    $$x^{C_r} \ge (r+x)x^{a_0+2r}+r^2.$$

    We will show that $C_r$ satisfies the claim. 

    \bigskip
    
    Define $0 \le k \le r$ to be maximal such that $V(G)$ admits a partition $\mathcal{A} \cup \{S_1, \dots, S_k, R\},$ such that the following conditions hold.
    
    \begin{itemize}
        \item $|\mathcal{A}| \le k(1/\e)^{a_0+2r} + k^2$ and each member of $\mathcal{A}$ is $\e$-sparse;
        \item $S_1, \dots, S_k$ is $(\e^{13r-8k-1},\e^{13r-8k-1},\e^5)$-full and for each $1 \le i \le k,$ the set $S_i$ is $\epsilon^{8(r-k)+1}$-sparse;
        \item for each $1 \le i \le k,$ we have that $|R| \le (\e/100) |S_i|$ and that $R$ is $(\e/6)$-dense to $S_i.$
    \end{itemize}

    Fix some partition as such.

\noindent\textbf{Claim 1:} We may assume that $|R| \ge 100(1/\e)^{r(1/\e)^{a_0}}.$

\noindent\textbf{Proof:} Suppose not. Then applying Lemma \ref{partition most} to $R$ with $m=a_0 + 2$ and $\alpha=\e,$ we may partition $R$ into at most $(1/\e)^{a_0+2r+1}$ sets which are $\e$-sparse, since the resulting leftover set $L$ has cardinality at most

$$100\e^{(r-1)(1/\e)^{a_0+2}-r(1/\e)^{a_0}} \le 100\e^{(3r-4)(1/\e)^{a_0}}<1$$

(since $a_0 \ge 2$). So then we have partitioned $V(G)$ into at most

$$k(1/\e)^{a_0+2r}+(1/\e)^{a_0+2r+1} + k^2 \le (1/\e)^{C_r}$$

sets, all of which are $\e$-sparse.

\bigskip

\noindent\textbf{Claim 2:} We must have that $k<r.$

\noindent\textbf{Proof:} Suppose not, so that $k=r.$ By Claim 1, the set $R$ is certainly nonempty, and so we may pick some vertex $v \in R.$ Consider the sequence $N(v) \cap S_1, \dots, N(v) \cap S_r.$ Since the sequence $S_1, \dots, S_r$ is $(\e^{5r-1},\e^{5r-1},\e^5)$-full and $|N(v) \cap S_i| \ge (\e/6)|S_i| \ge \e^4|S_i|$ for each $1 \le i \le r,$ it follows that the sequence $N(v) \cap S_1, \dots, N(v) \cap S_r$ is $(\e^{5r-5},\e^{5r-5},\e^5)$-full, and so by Lemma \ref{sequence}, there must be some copy of $K_r$ spanning these sets. But then combining this with the vertex $v$ gives a copy of $K_{r+1},$ contrary to assumption. This proves the claim.

\bigskip

We now proceed to derive a contradiction by showing that $k$ is not maximal. Apply Lemma \ref{sparse set} to $R$ with $\alpha = \e^{(1/\e)^{a_0}},$ calling the resulting set $Z_0.$ More generally, given $Z_{i-1} \subseteq R$ for some $1 \le i \le k,$ we may apply Lemma \ref{pairs lemma} to the pair $(S_i,Z_{i-1})$ with $\alpha,\beta,l$ replaced with $(\epsilon/6), \epsilon^{(1/\e)^{a_i}}, 13r$ respectively, calling the resulting pair $(X_i,Z_i).$ Having produced $Z_k$ in this way, define $Y_{k+1} = Z_k.$

Now for each $1\le i \le k,$ since $|X_i| \ge (\e/12)|S_i| \ge 2,$ we may find a partition $X_i = Y_i \cup Y_i'$ such that $|Y_i|,|Y_i'| \ge |X_i|/3.$ But then since $S_i$ is certainly $(\e/3)$-sparse and $|S_i\backslash Y_i| \ge |S_i|/3,$ the set $S_i \backslash Y_i$ is $\e$-sparse for each $1\le i \le k.$ Add each of these $k$ sets to $\mathcal{A}$ to produce $\mathcal{A}'.$

We now apply Lemma \ref{partition most} to the set $R \backslash Y_{k+1}$ with $m = a_0+1$ and $\alpha = \e.$ Add the resulting $\e$-sparse sets into $\mathcal{A}'$ to produce $\mathcal{A}'',$ and call the set of leftover vertices $R'.$ Then, for each $1 \le i \le k+1,$ recursively let $B_i$ be those vertices in $R' \backslash (B_1 \cup \dots \cup B_{i-1})$ which are $(\e/3)$-sparse to $Y_i,$ and define $\tilde{R} = R' \backslash (B_1 \cup \dots \cup B_{k+1}).$
    
\noindent\textbf{Claim 3:} For each $1 \le i \le k+1,$ we have that $|Y_i| \ge 100$ and that $|R'| \le (\e/300)|Y_i|.$

\noindent\textbf{Proof:} First, suppose that $1 \le i \le k.$ Then $|Y_i| \ge (\e/36)|S_i| \ge (\e/36)(100/\e)|R| = (25/9) |R| \ge 100$ (since certainly $|R| \ge 3$). Since $|R'| \le \e^{(r-1)(1/\e)^{a_0+1}} |R|,$ we have that 

$$|R'| \le (9/25)\e^{(r-1)(1/\e)^{a_0+1}}|Y_i| \le (\e/300)|Y_i|$$

(since $a_0 \ge 2$).

For the case $i = k+1,$ notice first that $|Y_{k+1}| \ge \e^{f(\e)}|R|,$ where

$$f(\e) = (r-1)(1/\e)^{a_0} + \sum_{j=1}^k (1/\e)^{a_j}(6/\e)^{26r} \le r(1/\e)^{a_0}.$$

So since $|R| \ge 100(1/\e)^{r(1/\e)^{a_0}},$ we have that $|Y_{k+1}| \ge 100.$

Now $|R'| \le \e^{(r-1)(1/\e)^{a_0+1} - f(\e)}|Y_{k+1}|.$ Noting that

\begin{align*}
    (r-1)(1/\e)^{a_0+1} - f(\e) &\ge 2(r-1)(1/\e)^{a_0} - r(1/\e)^{a_0} + 13r \\
    &\ge (r-2)(1/\e)^{a_0} +13r \ge 13r \ge 10,
\end{align*}

it follows that $|R'| \le \e^{10} |Y_{k+1}| \le (\e/300) |Y_{k+1}|.$

\bigskip 

By Claim 3, we may apply Lemma \ref{split} with $\alpha = \e/3$ to the pair $(Y_i,B_i)$ for each $1\le i \le k+1,$ calling the resulting partition $Y_i = T_i \cup T_i'.$ Produce $\tilde{\mathcal{A}}$ by adding the sets $T_i \cup B_i$ for each $1 \le i \le k+1$ to $\mathcal{A}'',$ noting that each of these sets are $\e$-sparse, because $T_i$ is $\e$-sparse, $B_i$ is $\e$-sparse to $T_i$ and $|B_i| \le \e |T_i|.$

Lastly, let $\tilde{S}_i = T_i'$ for each $1 \le i \le k+1.$

\bigskip

Observe now that the collection of sets $\tilde{\mathcal{A}} \cup \{\tilde{S}_1, \dots, \tilde{S}_{k+1}, \tilde{R}\}$ is a partition of $V(G),$ and that each set in $\tilde{\mathcal{A}}$ is $\e$-sparse.

\noindent\textbf{Claim 4:} $|\tilde{\mathcal{A}}| \le (k+1)(1/\e)^{a_0+2r} + (k+1)^2.$

\noindent\textbf{Proof:} This is immediate by noting that $|\tilde{\mathcal{A}}| \le |\mathcal{A}| + (1/\e)^{a_0+2r} +(2k+1).$

\bigskip

\noindent\textbf{Claim 5:} The sequence $\tilde{S}_1, \dots, \tilde{S}_{k+1}$ is $(\e^{13r-8(k+1)-1},\e^{13r-8(k+1)-1},\e^5)$-full.

\noindent\textbf{Proof:} Firstly, the sequence $\tilde{S}_1, \dots, \tilde{S}_k$ is $(\e^{13r-8(k+1)-1},\e^{13r-8(k+1)-1},\e^5)$-full, since $|\tilde{S_i}| \ge (\e/108) |S_i| \ge \e^8|S_i|$ for each $1 \le i \le k.$ Secondly, note that for each $1 \le i \le k,$ the pair $(\tilde{S}_i,\tilde{S}_{k+1})$ is $(9(\e/6)^{13r},3\e^{g(\e)},\e/12)$-full, where 

$$g(\e) = (1/\e)^{a_i} - \sum_{j=i+1}^k (1/\e)^{a_j}(6/\e)^{26r} \ge 13r.$$

Noting that $9(\e/6)^{13r} \le \e^{13r-4} \le \e^{13r-8(k+1)-1},$ that $3\e^{13r} \le \e^{13r-2} \le \e^{13r-8(k+1)-1}$ and that $\e/12 \ge \e^5,$ it follows that the pair $(\tilde{S}_i,\tilde{S}_{k+1})$ is $(\e^{13r-8(k+1)-1},\e^{13r-8(k+1)-1},\e^5)$-full for each $1\le i \le k.$

\bigskip

\noindent\textbf{Claim 6:} The set $\tilde{S}_i$ is $\epsilon^{8(r-k-1)+1}$-sparse for each $1 \le i \le k+1.$ 

\noindent\textbf{Proof:} Indeed, for $1 \le i \le k,$ this follows since $|\tilde{S}_i| \ge \e^8|S_i|.$ For $i=k+1,$ we see that, by construction, the set $\tilde{S}_{k+1}$ is $3\e^{h(\e)}$-sparse, where

$$h(\e) = (1/\e)^{a_0} - \sum_{j=1}^k (1/\e)^{a_j}(6/\e)^{26r} \ge 13r \ge 8r-5.$$ 

The claim follows since $3\e^{g(\e)} \le \e^{g(\e)-2} \le \e^{8r-7} \le \e^{8(r-k+1)+1}.$

\bigskip

\noindent\textbf{Claim 7:} $|\tilde{R}| \le (\e/100)|\tilde{S}_i|$ for each $1 \le i \le k+1.$

\noindent\textbf{Proof:} This is immediate, since $|\tilde{R}| \le |R'| \le (\e/300)|Y_i|$ and $|S_i^k| \ge |Y_i|/3$ for each $1 \le i \le k+1.$

\bigskip

But then we have shown that $k$ was not maximal, and so we are done.

\end{proof}
\section{Concluding remarks}

We have shown that for each integer $r \ge 2,$ there is a constant $C_r>0$ such that, for any $0<\e\leq 1/2,$ every $K_{r+1}$-free graph can be partitioned into at most $(1/\e)^{C_r}$ sets, all of which are $\e$-sparse. 

It is now known that the Erd\H{o}s-Hajnal and polynomial Rödl properties are equivalent. Fox, Nguyen, Scott and Seymour~\cite{Foxetalcographs} raised the question as to whether every graph has the strong polynomial Rödl property. We conjecture that the strong polynomial R\"odl is equivalent to the Erd\H{o}s-Hajnal property.

\begin{conjecture}
    Any graph $H$ which satisfies the Erd\H{o}s-Hajnal property also satisfies the strong polynomial R\"odl property.
\end{conjecture}

\bibliographystyle{alpha}
\bibliography{Rodlpartitions}
\end{document}